\documentclass[12pt]{amsart}
 \usepackage{mathrsfs}
 \usepackage {graphics}
\usepackage{amsmath}
\usepackage{amsfonts}
\usepackage{amssymb}
\usepackage{amsthm}
\usepackage{newlfont}

 \newtheorem{thm}{Theorem}[section]
 \newtheorem{cor}[thm]{Corollary}

 \theoremstyle{definition}
 
 \theoremstyle{remark}
 \newtheorem{rem}[thm]{Remark}
  \newtheorem{ques}[thm]{Question}
 
 \numberwithin{equation}{section}

\bigskip

\def\refn#1.#2{\expandafter\def\csname#1\endcsname{[#2]}}
\def\refnr#1.{\csname#1\endcsname}
\begin{document}
%
%
%
%
%
%
%
%
%
\title[Some Remarks On Essentially Normal Submodules]
 {Some Remarks On Essentially Normal  Submodules}
\author[R. G. Douglas]{Ronald G.~Douglas}

\address{%
Department of Mathematics,
Texas A\&M University, College Station, TX 77843, USA}

\email{rdouglas@math.tamu.edu}

\author[K. Wang]{Kai Wang}
\address{School of Mathematical Sciences,
Fudan University, Shanghai, 200433, P. R. China}
\email{kwang@fudan.edu.cn}
\thanks{ The second author was  supported by
  NSFC,   the Department of Mathematics  at Texas A\&M University and Laboratory of Mathematics for Nonlinear Science at Fudan University.}
\subjclass[2010]{ 47A13,
 46E22,46H25, 47A53.}

\keywords{Arveson conjecture, essentially normal submodules}


\begin{abstract} Given a $*$-homomorphism $\sigma: C(M)\to \mathscr{L}(\mathcal{H})$ on a Hilbert space $\mathcal{H}$ for a compact metric space $M$,
 a projection $P$ onto a subspace $\mathcal{P}$ in $\mathcal{H}$ is said to be essentially normal relative to $\sigma$ if $[\sigma(\varphi),P]\in \mathcal{K}$ for $\varphi\in C(M)$, where $\mathcal{K}$ is the ideal of compact operators on $\mathcal{H}$. In this note we  consider two notions of span      for essentially normal projections $P$ and $Q$, and investigate when they   are also essentially normal. First, we  show the representation theorem for two projections, and relate these results
 to  Arveson's conjecture for the closure of homogenous polynomial ideals on the Drury-Arveson space. Finally, we consider the relation between  the relative position of two essentially normal projections and the $K$ homology elements defined for  them.
\end{abstract}

\maketitle

\section{Introduction}
 Spurred by a question of Arveson \cite{Ar1,Ar2,Ar3,Ar4}, several researchers have been considering when certain submodules of various Hilbert modules of holomorphic functions on the unit ball in $\mathbb{C}^n$ are essentially normal.  In
particular, Guo and the second author showed in \cite{GW} that the closure of a principal homogenous polynomial ideal in the Drury-Arveson space in $\mathbb{B}^n$ is essentially normal.  More recently,
the authors have shown  in \cite{DW} that the closure of all principal polynomial ideals in the Bergman module  on the unit ball are essentially normal. Other results have been obtained by Arveson \cite{Ar5}, Douglas \cite{Dou,Dou1}, the first author and Sarkar \cite{DS},  Eschmeier \cite{Esch}, Kennedy \cite{Ke},  and Shalit \cite{Sha}.

The Arveson conjecture concerns the closure of an arbitrary  homogeneous polynomial ideal which, in general, is not singly  generated.  For the case of $n=1$, one knows that a pure hyponormal operator submodule is essentially normal if
  it is  finitely  generated.  The basis on which this result depends is the Berger-Shaw Theorem \cite{BS}.

For ideals that are not principal, or singly generated, the results in the several variable case are few. Guo in \cite{G} firstly proved    Arveson's conjecture in case of the dimension $n=2$. Guo and the second author established in \cite{GW2,GW} essential normality when $n=3$ or the dimension of
the zero variety of the homogeneous ideal is one or less, the opposite extreme, more or less, of the case of principal ideals.  There is also a result of Shalit  \cite{Sha} which holds for ideals having a "very nice" basis relative to the
norm. More recently, Kennedy \cite{Ke} extended that result in another direction, considering when the linear span of the closures of polynomial ideals is closed.
He gives some examples, but it would appear that not all non-principal ideals are  covered by this result. One should note that these results when the linear span of two essentially normal submodules is closed is implicit in the work of Arveson \cite[Theorem 4.4]{Ar5}.

In this note, we explore a  more general version of the question of when the linear span of two essentially normal submodules is also essentially normal.  We show that this result contains one aspect  of the
results of Shalit and Kennedy.

Our work does not depend on the special nature of the submodules; that is, we do not assume any connection with any underlying algebraic structure, only the fact that the linear span is closed.

There is more than one sense of the span of two submodules relevant in this context:  the first is the  obvious one defined to be  the closure   of the linear  span of
two submodules $\mathcal{P}$ and $\mathcal{Q}$, while the second one considers the span modulo the ideal of compact operators. If $P$ and $Q$ denote the orthogonal projections onto  $\mathcal{P}$ and $\mathcal{Q}$, respectively, then we will show that  this notion makes sense
if $0$ is an  isolated point in the essential spectrum of $P+Q$.

 We consider the first notion in Section 2, and  the results obtained are  based on the structure theorem for two projections. The latter notion of the essential span  is taken up in Section 3.
  We apply these results to the context of Arveson's conjecture and raise some questions. In particular, we assume that there is a $*$-homomorphism $\sigma$ of $C(M)$ for some compact metric space $M$ and the projections essentially commute with the range of $\sigma$.   Finally, in Section 4 we observe that an essentially normal projection determines an element of  the odd $K$-homology group for some compact subset of $M$ and  consider the relation of the K-homology
  elements defined by two essentially normal submodules and their sum.
\section{Refinement of the Two Projection Representation}\label{sec1}

\indent
Our results  in this section are based on refinements of the structure theorem for  two projections \cite{Br,H}.

Let $P$ and $Q$ be two projections on the Hilbert space $\mathcal{H}$.  Then there exist operators $S_1 : \mathcal{P} \to \mathcal{P}, S_2 : \mathcal{P}^\perp \to \mathcal{P}^\perp$ and $X: \mathcal{P}^\perp
\to \mathcal{P}$, where $\mathcal{P}={\text{ran }} {P}$   with
$0_{\mathcal{P}}\leq {S_1} \leq I_{\mathcal{P}}, {0_{\mathcal{P}^\perp}}  \leq{S_2}\leq I_{\mathcal{P}^\perp}$ and $ \|{X}\| \thinspace \leq 1$, such that
$$P=
\left(
\begin{array}{cc}
 I_{\mathcal{P}} & 0 \\0 & 0_{\mathcal{P}^\perp}
\end{array}
\right)
\text{ and }  Q=
\left(
\begin{array}{cc}
 S_1 & X \\X^* & S_2
\end{array}
\right).$$
Moreover, if we set $\mathcal{P}=\mathcal{P}_1 \oplus \mathcal{P}_2 \oplus \mathcal{P}_3$ and
$\mathcal{P}^\perp=\mathcal{Q}_1 \oplus \mathcal{Q}_2 \oplus \mathcal{Q}_3$, where
 $\mathcal{P}_1=\{x \in \mathcal{P}: \thinspace S_1{x}=0\}, \mathcal{P}_2=\{x \in \mathcal{P}: \thinspace S_1{x}=x\}$,
  $\mathcal{P}_3=\mathcal{P}\ominus(\mathcal{P}_1\oplus\mathcal{P}_2)$,  $\mathcal{Q}_2=\{x \in \mathcal{P}^\perp: \thinspace S_2{x}=x\}$,
$\mathcal{Q}_3=\{x \in \mathcal{P}^\perp: \thinspace S_2{x}=0\}$, and  $\mathcal{Q}_1=\mathcal{P}^\perp\ominus(\mathcal{Q}_2\oplus\mathcal{Q}_3)$, then we have
\begin{align*}
S_1 & =
\left(
\begin{array}{ccc}
 0_{\mathcal{P}_1} & 0 & 0 \\
 0 & I_{\mathcal{P}_2} & 0 \\
 0 & 0 & {S'_1}
\end{array}
\right) \in \mathscr{L}(\mathcal{P}_1 \oplus \mathcal{P}_2 \oplus \mathcal{P}_3)  \text{ with } S'_1\in \mathscr{L}(\mathcal{P}_3),\\
S_2 & =
\left(
\begin{array}{ccc}
 {S'_2} & 0 & 0 \\
 0 & I_{\mathcal{Q}_2} & 0 \\
 0 & 0 & 0_{\mathcal{Q}_3}
\end{array}
\right) \in \mathscr{L}(\mathcal{Q}_1 \oplus \mathcal{Q}_2 \oplus \mathcal{Q}_3)  \text{ with } S'_2\in \mathscr{L}(\mathcal{Q}_1), \text{ and}\end{align*}
\begin{align*}
X & =
\left(
\begin{array}{ccc}
 0 & 0 & 0 \\
 0 & 0 & 0 \\
 X' & 0 & 0
\end{array}
\right) \in \mathscr{L}(\mathcal{Q}_1 \oplus \mathcal{Q}_2 \oplus \mathcal{Q}_3, \mathcal{P}_1 \oplus \mathcal{P}_2 \oplus \mathcal{P}_3) \text{ with } X'\in \mathscr{L}(\mathcal{Q}_1, \mathcal{P}_3).
\end{align*}
These results are all straightforward.

Further, using matrix computations and the fact that $Q^2=Q=Q^\ast$,  one shows that there exists an isometry $V$ from $\mathcal{Q}_1$ onto $\mathcal{P}_3$ such that $V^* S'_1 V=I_{\mathcal{Q}_1}-S'_2$. We refer the reader to \cite{H} for a detailed argument. Therefore, we have derived the standard model for two projections.
\begin{thm} Two projections  $P$
and $Q$ on a Hilbert space $\mathcal{H}$ are determined by

\indent(1) a decomposition $\mathcal{H}=\mathcal{H}_0 \oplus \mathcal{H}_1 \oplus \mathcal{H}' \oplus \mathcal{H}' \oplus \mathcal{H}_2 \oplus \mathcal{H}_3$, and

\indent(2) a positive contraction $S \in \mathcal{L}(\mathcal{H}')$ with $\{0, 1\}$ not in its point spectrum.

In this case, one has
$$
 P=
\left(
\begin{array}{cccccc}
 I_{\mathcal{H}_0} & 0 & 0 & 0 & 0 & 0\\
 0 & I_{\mathcal{H}_1} & 0 & 0 & 0 & 0 \\
 0 & 0 & I_{\mathcal{H}'} & 0 & 0 & 0 \\
 0 & 0 & 0 & 0_{\mathcal{H}'} & 0 & 0 \\
 0 & 0 & 0 & 0 & 0_{\mathcal{H}_2} & 0 \\
 0 & 0 & 0 & 0 & 0 & 0_{\mathcal{H}_3} \\
\end{array}
\right) \text{ and }
Q=
\left(
\begin{array}{cccccc}
 0_{\mathcal{H}_0} & 0 & 0 & 0 & 0 & 0\\
 0 & I_{\mathcal{H}_1} & 0 & 0 & 0 & 0\\
 0 & 0 & S & X & 0 & 0\\
 0 & 0 & X & I_{\mathcal{H}'}-S  & 0 & 0       \\
  0 & 0 & 0 & 0 & I_{\mathcal{H}_2} & 0   \\
  0 &  0 & 0 & 0 & 0 & 0_{\mathcal{H}_3}
\end{array}
\right),
$$ where $X=\sqrt{S(I_{\mathcal{H}'}-S)} \in \mathcal{L}(\mathcal{H}')$.\end{thm}
\begin{proof}
Again,  the representation results follow from standard matrix computations.\end{proof}
The following question   happens frequently  in many concrete problems in   operator theory.
\begin{ques} When is $\mathcal{P}+\mathcal{Q}$ closed in $\mathcal{H}$, where $\mathcal{P}=\text{ran }P$ and $\mathcal{Q}=\text{ran }Q$?\end{ques}

Note we have:

\begin{align*}
\mathcal{P} & =
\left\{
\left(
\begin{array}{c}
 x_0 \\
 x_1 \\
 x' \\
 0 \\
 0 \\
 0
\end{array}
\right)
: x_0 \in \mathcal{H}_0, x_1 \in \mathcal{H}_1 , x' \in \mathcal{H}' \right\} \text{ and }\end{align*}
\begin{align*}
\mathcal{Q} & =
\left\{
\left(
\begin{array}{c}
 0 \\
 x_1 \\
 Sx'+Xy' \\
 Xx'+(I_{\mathcal{H}'}-S)y' \\
 x_2 \\
 0
\end{array}
\right): x_1 \in \mathcal{H}_1 , x', y' \in \mathcal{H}' , x_2 \in \mathcal{H}_2 \right\}.
\end{align*}

Therefore,
\begin{equation*}
\mathcal{P}+\mathcal{Q}=
\left\{
\left(
\begin{array}{c}
 x_0 \\
 x_1 \\
 x' \\
 z \\
 x_2 \\
 0
\end{array}
\right) : x_0 \in \mathcal{H}_0 , x_1 \in \mathcal{H}_1, x' \in \mathcal{H}', z \in \text{ ran} X +\text{ ran}(I_{\mathcal{H}'}-S), x_2 \in \mathcal{H}_2 \right\}.
\end{equation*}
This implies that $\mathcal{P}+\mathcal{Q}$ is closed if and only if $\text{ran }X+ \text{ ran}(I_{\mathcal{H}'}-S)$ is closed.  Since $X=\sqrt{S(I_{\mathcal{H}'}-S)}$, we have  that
$\text{ran }X \subseteq \text{ran}(I_{\mathcal{H}'}-S)^{\frac{1}{2}}  $. Moreover, by the spectral theorem   for the positive contraction $S$, one sees that $\sqrt{S}+\sqrt{I_{\mathcal{H}'}-S}$ is invertible on $\mathcal{H}'$. This implies that $\text{ran }X +\text{ran}(I_{\mathcal{H}'}-S) \supseteq \text{ran}(I_{\mathcal{H}'}-S)^{\frac{1}{2}}$. Therefore, $\text{ran }X+\text{ran}(I_{\mathcal{H}'}-S)$ is closed if and only if $\text{ran}(I_{\mathcal{H}'}-S)^{\frac{1}{2}}$ is closed.   Since $1$ is not in the point spectrum of $S$, it
follows from the spectral theorem  that $\text{ran}(I_{\mathcal{H}'}-S)^{\frac{1}{2}}$ is closed if and only if $1$ is not in the spectrum of $S$. Hence we have the following result.
\begin{thm} For two projections $P$ and $Q$ on the Hilbert space $\mathcal{H}$ with $\mathcal{P}=\textrm{ ran} P$ and $\mathcal{Q}=\textrm{ ran} Q$, the linear span $\mathcal{P}+\mathcal{Q}$ is closed if and only if $1\notin\sigma(S)$, or equivalently, $\sigma(PQP)\cap (\varepsilon,1)=\phi$ for some $0<\varepsilon<1$, where $S$ is the same as in Theorem 2.1. Moreover, in the case that $\mathcal{R}=\mathcal{P}+\mathcal{Q}$ is closed,
the projection $R$ onto $\mathcal{R}$ has the form
$$
 R=
\left(
\begin{array}{cccccc}
 I_{\mathcal{H}_0} & 0 & 0 & 0 & 0 & 0\\
 0 & I_{\mathcal{H}_1} & 0 & 0 & 0 & 0 \\
 0 & 0 & I_{\mathcal{H}'} & 0 & 0 & 0 \\
 0 & 0 & 0 & I_{\mathcal{H}'} & 0 & 0 \\
 0 & 0 & 0 & 0 & I_{\mathcal{H}_2} & 0 \\
 0 & 0 & 0 & 0 & 0 & 0_{\mathcal{H}_3} \\
\end{array}
\right)$$
\end{thm}
\begin{proof}
Only the last statement remains to be proved and that follows from the fact that $1 \notin \sigma( S)$ if and only if $0\notin\sigma( I_{\mathcal{H}'}-S)$, which implies $\text{ ran} (I_{\mathcal{H}'}-S)=\mathcal{H}'$.
\end{proof}
A nearly immediate consequence of the representation theorem is the following characterization   of the $C^*$-algebra $\mathcal{A}(P,Q,I)$ generated by projections $P$, $Q$ and the identity operator on the Hilbert space $\mathcal{H} $. This is usually attributed  to Dixmier \cite{Di}.
\begin{thm}
Let  $P$ and $Q$ be the projections onto the subspaces $\mathcal{P}$ and $\mathcal{Q}$  of the Hilbert space $\mathcal{H}$, respectively. If $\mathcal{P}\cap \mathcal{Q}=\mathcal{P}^\perp\cap \mathcal{Q}=\mathcal{P}\cap \mathcal{Q}^\perp=\mathcal{P}^\perp\cap \mathcal{Q}^\perp=\{0\},$ then  $  \mathcal{A}(P,Q,I)$ is $*-$algebaically  isomorphic to a *-subalgebra $\mathcal{C} $ of $M_2(C(M)), $  where $M=\sigma(PQP)$, $M_2(C(M))$ denotes the algebra of two by two matrices  with entries  in $C(M)$ and $$\mathcal{C}=\{\left(
\begin{array}{cc }
\phi_{11} & \phi_{12}  \\
 \phi_{21} &  \phi_{22}
\end{array}
\right)\in M_2(C(M)):\phi_{12}(i)=\phi_{21}(i)=0, \text{if } i=0,1 \text{ and } i\in M\}.$$

\end{thm}
\begin{proof}Applying the spectral theorem  to the operators $I_{\mathcal{P}}$ and $S$, one obtains the correspondence  from which the result follows: $$
 P=
\left(
\begin{array}{cc }
 I_{\mathcal{P}} & 0  \\
 0 &   0
\end{array}
\right) \thicksim \left(\begin{array}{cc }
 1 & 0  \\
 0 &   0
\end{array}
\right)\in M_2(C(M))\text{ and }$$
$$
Q=\left(
\begin{array}{cc }
 S & \sqrt{S(1-S)}  \\
 \sqrt{S(1-S)} &  1-S
\end{array}
\right)\thicksim \left(\begin{array}{cc }
 \chi & \sqrt{\chi(1-\chi)}  \\
 \sqrt{\chi(1-\chi)} &  1-\chi
\end{array}
\right)\in M_2(C(M)),
$$
where $1$ and $\chi$ denote the functions on $M$ defined by $1(x)=1$ and $\chi (x)=x$ for $x\in M$. The fact that the functions $\phi_{12}$ and $\phi_{21}$ in the definition of $\mathcal{C}$ vanish at $0,1\in M$ follows from the fact that the function $\sqrt{\chi (1-\chi)}$ does.\end{proof}
We now use the characterization of the $C^*$-algebra generated by two projections to get our first result on the essential normality of the projection onto the linear span when it is closed.
\begin{thm}
For two projections $P$ and $Q$ on the Hilbert space $\mathcal{H}$, if $\mathcal{R}=\text{ran}P+\text{ran}Q$ is closed, then the   $C^*$-algebra $\mathcal{A}(P,Q,I_\mathcal{H})$ generated by $P, Q$ and the identity operator $I_\mathcal{H}$ contains    the
projection  $R$ onto the subspace $\mathcal{R}$.
\end{thm}
\begin{proof}Using a direct matrix computation, one sees  that the  operator $
 P+(I-P)Q(I-P)$, which is   in  the $C^*$-algebra $\mathcal{A}(P,Q,I)$, has the form
$$
 P+(I-P)Q(I-P)=
\left(
\begin{array}{cccccc}
 I_{\mathcal{H}_0} & 0 & 0 & 0 & 0 & 0\\
 0 & I_{\mathcal{H}_1} & 0 & 0 & 0 & 0 \\
 0 & 0 & I_{\mathcal{H}'} & 0 & 0 & 0 \\
 0 & 0 & 0 & I_{\mathcal{H}'}-S & 0 & 0 \\
 0 & 0 & 0 & 0 & I_{\mathcal{H}_2} & 0 \\
 0 & 0 & 0 & 0 & 0 & 0_{\mathcal{H}_3} \\
\end{array}
\right).$$
This implies that  $\sigma( P+(I-P)Q(I-P))\subseteq \{0\}\cup [ \varepsilon,1] $ for some $0<\varepsilon<1$. Since $[\varepsilon,1]\cap  \sigma( P+(I-P)Q(I-P))$
is an open and closed subset of the spectrum $\sigma( P+(I-P)Q(I-P))$, it follows from the spectral theorem that the spectral projection ${\mathbf{1}}_{[\varepsilon,1]} (P+(I-P)Q(I-P))$  is in $\mathcal{A}(P,Q,I)$, which leads to the desired result since  $\mathbf{1}_{[\varepsilon,1]} (P+(I-P)Q(I-P)) = {R}$.
\end{proof}
We now relate the representation result to  a question in the context of Arveson's conjecture. We will provide a more precise  statement   in Section 4.
\begin{thm}
Suppose $\sigma:\thinspace C(M)\to \mathcal{L}(\mathcal{H})$ is a $\ast$-homomorphism  for some compact metric space $M$, and $P,Q$ are projections on the Hilbert space $\mathcal{H}$ such that the commutators $[\sigma(\varphi),P]\in \mathcal{K}$ and  $[\sigma(\varphi),Q]\in \mathcal{K}$ for
$\varphi \in C(M)$, where $\mathcal{K}$ denotes the ideal of compact operators on $\mathcal{H}$.  If $\mathcal{R}=\text{ran}P+\text{ran}Q $ is closed and $R$ is the projection onto $\mathcal{R}$, then $[\sigma(\varphi),R]\in \mathcal{K}$ for
$\varphi \in C(M)$.
\end{thm}
 \begin{proof}
 Using an elementary $C^*$-algebra argument, one     shows that $[\sigma(\varphi),T]\in \mathcal{K}$ for any operator $T\in \mathcal{A}(P,Q,I)$. Combining this fact with Theorem 2.5, one obtains the desired result.
 \end{proof}
\begin{cor}
With the same hypotheses, the projection $\widetilde{R}$ onto $\mathcal{P}\cap \mathcal{Q}$ essentially commutes with the range of $\sigma$.
\end{cor}
\begin{proof}This is an  immediate consequence  of Theorem 1 in \cite{Dou1} and the exact sequence
$$0\to\widetilde{R} \stackrel{i}{\to}  \mathcal{P}\oplus {\mathcal{Q}}\stackrel{j}{ \to} \mathcal{R}\to 0,$$ where $i(r)=(r,-r)$, and $j(p,q)=p+q$ for $r\in \widetilde{R},p\in \mathcal{P},q\in\mathcal{Q}$.
\end{proof}

\begin{rem} In both the theorem and corollary,  $C(M)$ can be replaced by any $C$*-subalgebra of $\mathscr{L}(\mathcal{H})$.\end{rem}


\begin{rem}  These results are related to a theorem of Arveson \cite[Theorem 4.4]{Ar5} and the more recent work of Kennedy \cite{Ke} in which essential normality is replaced by $p$-essential normality, where the commutators are assumed to be in the Schatten $p-$class for $1\leq p< \infty$.  If one examines the proof of Theorem 2.5 more closely,    the preceding
arguments can be refined to obtain analogous results for $p$-essential normality. Basically, this is true because the functional calculus which yields the spectral projection  $\mathbf{1}_{[\varepsilon,1]} (P+(I-P)Q(I-P))$ can be approximated on a neighborhood of  the spectrum with analytic functions. \end{rem}

We can extend these results somewhat using the following reduction which is essentially algebraic.
\begin{thm}
Suppose $\mathcal{P}$ and $\mathcal{Q}$ are subspaces of the Hilbert space $\mathcal{H}$ and $\mathcal{R}^\sharp$ is a subspace of $\mathcal{P}\cap\mathcal{Q}$. Then $\mathcal{P}+\mathcal{Q}$ is closed if and only if $\mathcal{P}/  {\mathcal{R}^\sharp}+ \mathcal{Q}/  {\mathcal{R}^\sharp}$ is closed in $\mathcal{H}/  {\mathcal{R}^\sharp}$.
\end{thm}
 \begin{proof} It follows from  the fact that $ \mathcal{P}/  {\mathcal{R}^\sharp}+ \mathcal{Q}/  {\mathcal{R}^\sharp}= (\mathcal{P}+ \mathcal{Q})/  {\mathcal{R}^\sharp} $
 and   the fact that  for any linear manifold $\mathcal{L}$ containing $\mathcal{R}^\sharp$, $\mathcal{L}/ \mathcal{R}^\sharp$ is closed if and only if $\mathcal{L}$ is closed.
 \end{proof}
 \begin{cor} With the same hypotheses, the closeness of $\mathcal{P}/  {\mathcal{R}^\sharp}+\mathcal{Q}/  {\mathcal{R}^\sharp}$ is equivalent to the closeness of $\mathcal{P}/  (\mathcal{P}\cap \mathcal{Q})+\mathcal{Q}/  (\mathcal{P}\cap \mathcal{Q}) $.
 \end{cor}
 \begin{proof}
Both of these statements are equivalent to  $\mathcal{P}+\mathcal{Q}$ being closed in $\mathcal{H}$.
 \end{proof}
\section{Essential Span of Subspaces}\label{sec3}

The following question and corresponding result are important for considering the notion of essential span in this section.
\begin{ques} When do two projections $P$ and $Q$ on a Hilbert space $\mathcal{H}$ almost commute; that is, when is $[P,Q]\in \mathcal{K}(\mathcal{H})?$\end{ques}
 Using   the representation theorem for $P$ and $Q$ above, we see that $[P,Q]\in \mathcal{K}$ if and only if $X\in \mathcal{K}$ and so we have the following result.

\begin{thm}
 For projections $P$ and $Q$ onto subspaces $\mathcal{P}$ and $\mathcal{Q}$, respectively, on a Hilbert space $\mathcal{H}, \thinspace [P,Q]\in \mathcal{K}$ if and only if ${\sigma_e}(S)\subset \{0,1\}$. Moreover, $PQ\in\mathcal{K}$ if and only if $S\in\mathcal{K}$ and $\dim \mathcal{P}\cap \mathcal{Q} < \infty$  in the representation appearing in Theorem 2.1.
\end{thm}
\begin{proof}
The proof  follows  from a matrix calculation in the above representation theorem  which shows that  $[P,Q]\in\mathcal{K}$ if and only if $X=\sqrt{S(I_{\mathcal{H}'}-S)}$ is compact. For $PQ\in\mathcal{K}$, it is necessary and sufficient for $S$ and $I_{\mathcal{H}_1}$ to be compact.
\end{proof}

If $P$ and $Q$ are projections on the Hilbert space $\mathcal{H}$, then another notion of the span of the ranges of $P$ and $Q$ is relevant when considering questions of essential normality, which involves
the images of $P$ and $Q$ in the Calkin  algebra.  If $0$ is an isolated point in the essential spectrum, ${\sigma_e}(P+Q)$, of $P+Q$, or $0\notin {\sigma_e}(P+Q)$, then the image in the Calkin algebra of the  spectral projection,  ${P} \bigvee_e {Q}$, for
$[\varepsilon, \infty]$, where $(0, \varepsilon)\cap {\sigma_e}(P+Q)=\phi$, can be thought of as the "essential span" of $\text{ran}P$ and $\text{ran}Q$. (Note that the image of this spectral projection in the Calkin algebra does not depend on $\varepsilon$ whenever $(0, \varepsilon)\cap {\sigma_e}(P+Q)=\phi$.)  One result related to this notion  is
the following.

\begin{thm}
 If $[P,Q]\in\mathcal{K}$, then $0$ is isolated in ${\sigma_e}(P+Q)$.  Moreover, if $P$ and $Q$ almost commute with a $C^*$-algebra $ \mathfrak{A}$, then so does any   projection on $\mathcal{H}$ with the image ${P} \bigvee_e {Q}$ in the Calkin algebra.
\begin{proof}
 Considering the standard model for two projection in Section 2, one sees that $[P,Q]\in\mathcal{K}$ implies that $X$ is compact and  ${\sigma_e}(P+Q)\subseteq \{0,1,2\}$.  This implies that $[\varepsilon,2]$ is an open and closed subset of ${\sigma }(P+Q)$ for some $0<\varepsilon<1 $ and hence $\mathbf{1}_{[\varepsilon,2]} (P+ Q )$  is in $\mathcal{A}(P,Q,I)$, where $\mathcal{A}(P,Q,I)$ is the $C^*$-algebra generated by $P,Q$ and the identity operator $I$. Therefore,   its image in the Calkin algebra,   ${P} \bigvee_e {Q}$, commutes with the image of $ \mathfrak{A}$, which completes the proof.
\end{proof}
\end{thm}

One thing one needs to be clear on is that the image of  ${P} \bigvee {Q}$ in the Calkin algebra and ${P} \bigvee_e {Q}$  are not necessarily the same.  Consider, for example,  the subspaces  $\mathcal{P}=\overline{span}\{ e_n \oplus 0: n \in \mathbb{N}\}$ in
$\ell^2 \oplus \ell^2$ and $\mathcal{Q}=\overline{span}\{ e_n \oplus \frac{1}{n}e_n: n \in \mathbb{N}\}$.  These subspaces have the images of $\pi(P \bigvee Q)$ and $P \bigvee_e Q$ in the Calkin algebra which are the images of the projections onto
$\ell^2 \oplus \ell^2$ and $\ell^2 \oplus (0)$, respectively.  Note that in  this  case  $\mathcal{P}+\mathcal{Q}$ is not closed, which is the key as the following result shows.

\begin{thm}
 Let $P,Q$ be the projections onto the  subspaces $\mathcal{P}$ and $\mathcal{Q}$ of a Hilbert space  $\mathcal{H}$, respectively. Then  $\mathcal{P}+\mathcal{Q}$ is closed if and only if  $\pi( {P}\bigvee  {Q})=  {P}\bigvee_e  {Q} $.\end{thm}
\begin{proof}
We first suppose that $\mathcal{P}+\mathcal{Q}$ is closed. Using the notation in Theorem 2.1,  we have that
$$ P+Q=
\left(
\begin{array}{cccccc}
 I_{\mathcal{H}_0} & 0 & 0 & 0 & 0 & 0\\
 0 & 2I_{\mathcal{H}_1} & 0 & 0 & 0 & 0\\
 0 & 0 & I_{\mathcal{H}'}+S & \sqrt{S(I_{\mathcal{H}'}-S)} & 0 & 0\\
 0 & 0 & \sqrt{S(I_{\mathcal{H}'}-S)} & I_{\mathcal{H}'}-S  & 0 & 0       \\
  0 & 0 & 0 & 0 & I_{\mathcal{H}_2} & 0   \\
  0 &  0 & 0 & 0 & 0 & 0_{\mathcal{H}_3}
\end{array}
\right).$$
By Theorem 2.3  we know that $1\notin \sigma(S)$ when $\mathcal{P}+\mathcal{Q}$   is closed. Applying the spectral theorem to the operator $S$, one obtains that $0$ is isolated in    $\sigma(P+Q)$. This implies that  the notion ${P}\bigvee_e  {Q} $ makes sense and, in fact, it is the image in the Calkin algebra of the projection onto $\text{ran}(P+Q)$. Moreover, by the above representation of $P+Q$, one sees that
$$\text{ran}(P+Q)=\mathcal{H}_0\oplus\mathcal{H}_1\oplus \mathcal{H}'\oplus\mathcal{H}'\oplus \mathcal{H}_2 =\mathcal{P}+\mathcal{Q}.$$
It follows that ${P}\bigvee_e  {Q} $ is the image   of the projection onto $\mathcal{P}+\mathcal{Q}$.

On the other hand, in case that  $\pi( {P}\bigvee  {Q})=  {P}\bigvee_e  {Q} $, there exists $0<\varepsilon<1$  such that $(0,\varepsilon)\cap \sigma_e(P+Q)=\phi$ and ${P}\bigvee  {Q}-\mathbf{1}_{[\varepsilon,\infty]} (P+ Q )$ is a finite dimensional projection. Applying the spectral theorem for $S$ to the matrix representation of $P+ Q $, one sees that the spectral projection of $S$ for $(1-\varepsilon,1)$ is also  a finite dimensional projection. Combing this fact with that $1$ is not   in the point  spectrum of $S$, we have that $1\notin \sigma(S)$, which leads to the desired result using Theorem 2.3.
\end{proof}


While it seems  inconceivable  that $[p]+[q]$ is  always  closed for polynomials $p$ and $q$ in $\mathbb{C}[z_1,\cdots,z_n]$; here $[\cdot]$ denotes the closure  in the Hardy, Bergman  or Drury-Arveson modules on the unit ball,   it seems quite possible that the projections onto
$[p]$ and $[q]$ always almost commute. One thing making the answering of such a question difficult is the fact that $[p]\cap [q]$ is always large containing $[pq]$. One possible way to circumvent this problem might be to consider the quotient modules $[p]^\perp$ and $[q]^\perp$. We'll have something more to say about them in the next section.

Another possibility to handle the fact that $[p]\cap [q]$ is large might be to use  Theorem 2.10 to reduce the matter to $[p]/([p]\cap [q])$ and $[q]/([p]\cap [q]).$ In this case, $[p]/([p]\cap [q])$ and $[q]/([p]\cap [q]) $ are semi-invariant modules. We will obtain a result using this approach in the following section.

\section{Locality of Essentially Normal Projection}

Let $M$ be a compact metric space and $\sigma: C(M)\to \mathscr{L}(\mathcal{H})$ be a $*$-homomorphism for a  Hilbert space $\mathcal{H}$. We   say that a projection $P$ on $\mathcal{H}$ is essentially normal relative to $\sigma$ if $[\sigma(\varphi),P]\in\mathcal{K}$ for any $\varphi\in C(M)$. This implies that the map $\sigma_P: \varphi\to \pi(P\sigma(\varphi)P)\in\mathscr{Q}(\mathcal{H}) $ into the Calkin algebra  $\mathscr{Q}(\mathcal{H})=\mathscr{L}(\mathcal{H})/\mathcal{K}(\mathcal{H})$ is  a $*-$homomorphism. Hence, there exists
a compact subset $M_P$ of $M$ such that the following diagram commutes:
 \begin{equation*}
\begin{array}{ccc}
C(M) & \stackrel{\sigma}{\longrightarrow}   &  \mathscr{L}(\mathcal{H}) \\
    \downarrow   &          & \downarrow  \\
C(M_P) & \stackrel{\hat{\sigma}_P}{\longrightarrow}    &  \mathscr{Q}(P\mathcal{H}) \\
 \end{array}.
\end{equation*}
 Here the vertical arrow on the left is defined by restriction; that is,  $\varphi\to \varphi|_{M_P}$, and the one on the right is the compression to $\text{ran } P$ followed by the map onto the Calkin algebra. Therefore, using \cite{BDF}, one knows that $(\sigma, P)$ defines an element $[\sigma,P]\in K_1(M_ {P})$. An interesting question concerns
the relation of elements  $[\sigma,P]$ and $[\sigma,Q]$ for two essentially normal projections $P$ and $Q$ relative to $\sigma$.

 Now this relationship can't be too simple. In particular, consider the representation $\tau$ of $C(clos \mathbb{B}^n)$ in $L^2(\mathbb{B}^n)$ and the projection $P$ onto the Bergman space $L^2_a(\mathbb{B}^n)$. For $p\in \mathbb{C}[z_1,\cdots,z_n]$, one knows \cite{DW} that the projection $Q_p$ of $L^2(\mathbb{B}^n)$
 onto the closure $[p]$ of the ideal $(p)$ in $ \mathbb{C}[z_1,\cdots,z_n]$ generated by $p$ is essentially normal; that is, $[\tau(\varphi),Q_p]\in\mathcal{K}$ for $\varphi\in C(clos \mathbb{B}^n)$. Further,  we have that $R_p=P-Q_p $  is also essentially normal and    $M_{[\tau,R_p]}\subseteq Z(p)\cap \partial \mathbb{B}^n$, where $Z(p)$ is the zero variety of the polynomial $p$.   It follows that the image of $[\tau, R_p]\in K_1(\partial \mathbb{B}^n)$ is zero since $Z(p)\cap \partial \mathbb{B}^n$ is a proper subset of $ \partial \mathbb{B}^n $. Therefore, one has $[\tau,P]=[\tau, Q_p]$ for every polynomial $p\in\mathbb{C}[z_1,\cdots,z_n]$. Hence, there is a great variety of essentially normal projections defining the same element in $K_1(\partial \mathbb{B}^n).$

However, we do have a result   for what happens at the opposite extreme.
\begin{thm}
Suppose that $P$ and $Q$ are essentially normal projections on the Hilbert space $\mathcal{H}$ for the $*-$homomorphism $\sigma: C(M)\to \mathscr{L}(\mathcal{H})$ for some compact space $M$. If $M_P\cap M_Q=\phi$, then $PQ\in \mathcal{K}.$
\end{thm}
\begin{proof}
By the assumption that $P$ and $Q$ are essentially normal relative to $\sigma$, one sees that the operator $PQ$ almost intertwines the two representations $\sigma|_{C(M_P)}$ and $\sigma|_{C(M_Q)}$; that is, one has that $P \sigma(\varphi)P (PQ)-(PQ) Q\sigma(\varphi)Q\in \mathcal {K}$ for $\varphi\in C(M)$. Thus, in the Calkin algebra, if $\varphi\in C(M)$ satisfies  $\varphi|_{M_P}\equiv 1$ and  $\varphi|_{M_Q}\equiv 0$, we obtain that $\pi (Q \sigma(\varphi)Q)\pi(PQ)=\pi(PQ)\pi(P\sigma(\varphi)P)$. But, $\pi (Q \sigma(\varphi)Q)=0$ and $\pi (P \sigma(\varphi)P)=\pi(P)$, this means that $\pi(PQ P)=0$, which implies $PQ\in\mathcal{K}$ and completes the proof.
\end{proof}
We can use this theorem to obtain a partial result concerning the relation of the projections onto $[p]$ and $[q]$ for $p,q\in\mathbb{C}[z_1,\cdots,z_n]$.
\begin{cor} For two polynomials $p,q\in\mathbb{C}[z_1,\cdots,z_n]$, let $P$ and $Q$ be the projections  onto the submodule $\mathcal{P}=[p],\mathcal{Q}=[q]$ on $L^2_a(\mathbb{B}^n)$, respectively.
If $p,q$ satisfy  $Z(p)\cap Z(q) \cap \partial \mathbb{B}^n=\phi$, then we have that  $[P,Q]\in \mathcal{K}$.
\end{cor}
\begin{proof} Note that $I-P,I-Q$ are the projections onto the quotient modules $\mathcal{P}^\perp$ and $\mathcal{Q}^\perp$,  respectively.
Using the notation in the above, by \cite{DW} we know that $M_{I-P}\subseteq Z(p)\cap \partial \mathbb{B}^n$ and  $M_{I-Q}\subseteq Z(q)\cap \partial \mathbb{B}^n$. It follows from the hypothesis  that $M_{I-P}\cap M_{I-Q}=\phi$. By Theorem 4.1 we have that $(I-P)(I-Q) $ and $(I-Q)(I-P) $ are compact. Therefore, one sees that
$[P,Q]\in\mathcal{K}$, which completes the proof.
\end{proof}
We can extend this result using Theorem 2.10 as follows.
\begin{cor} For   polynomials $p,q,r\in\mathbb{C}[z_1,\cdots,z_n]$ with $Z(p)\cap Z(q)\cap\partial  \mathbb{B}^n=\phi$, let $\mathcal{P}=[pr]$ and $\mathcal{Q}=[qr]$ be the submodules in $L_a^2(\mathbb{B}^n)$. Then one has $[P,Q]\in\mathcal{K}$, where $P,Q$ are the projections onto $\mathcal{P}$ and $\mathcal{Q}$, respectively.
\end{cor}
\begin{proof}
One can generalize the argument in \cite{DW} to show that $$M_{[pr]^\perp/[r]^\perp}\subseteq Z(p)\cap \partial \mathbb{B}^n  \, and \, M_{[qr]^\perp/[r]^\perp}\subseteq Z(q)\cap \partial \mathbb{B}^n.$$   By   Theorem 4.1, this implies that $(R-P)(R-Q)$ and $(R-Q)(R-P) $ are compact,  where $R$ is the projection onto the submodule $[r].$ This means that $[P,Q]=[R-P,R-Q]\in\mathcal{K}$, which completes the proof.
\end{proof}
Another example of the application of the notion of the locality of essentially normal projection is the following result which is more or less the opposite situation of the previous theorem.
\begin{thm}
Assume that $\sigma: C(M) \to \mathscr{Q}(\mathcal{H})$ is a $*-$homomorphism on
 the Hilbert space $\mathcal{H}$ for a compact metric space $M$, and $P$ and $Q$ are two essentially normal projections such that $\mathcal{P}\cap \mathcal{Q}^\perp=\mathcal{P}^\perp \cap \mathcal{Q}=\{0\}$, where   $\mathcal{P}=\text{ran} P$ and $\mathcal{Q}=\text{ran}Q$.  If
 $\mathcal{P}+\mathcal{Q}^\perp$ is closed, then $M_P=M_Q$ and $[\hat{\sigma}_P]=[\hat{\sigma}_Q]\in K_1(M_P)$.
\end{thm}
\begin{proof} 
In the representation theorem for $P,Q$, the spaces $\mathcal{H}_0$ and $\mathcal{H}_2$ are $\{0\}$ by assumption and we can write $\mathcal{P}=\mathcal{H}_1\oplus \mathcal{P}' $ and $\mathcal{Q}=\mathcal{H}_1\oplus \mathcal{Q}'$ corresponding to  ${P}'=P-I_{\mathcal{H}_1}$ and ${Q}'=Q-I_{\mathcal{H}_1}$. As in the proof of Theorem 4.1, the image $\pi(P'Q')$ of $P'Q'$ in the Calkin algebra intertwines the operators $\pi (P' \sigma(\varphi) P')$ and  $\pi (Q' \sigma(\varphi) Q')$.  Using Theorem 2.3 and the assumption $\mathcal{P}+\mathcal{Q}^\perp$ is closed, we have  $0\notin\sigma(S)=\sigma(P'Q'P')$. Combining this with the fact $\ker P'Q' =\mathcal{P}'^\perp \cap \mathcal{Q}'=\{0\}$, one sees that $P'Q':\mathcal{Q}'\to\mathcal{P}'$ is invertible.  Therefore, using the polar decomposition in the Calkin algebra, one sees that $M_P=M_Q$ and that the $K_1$ elements are equal.
\end{proof}
There would seem to be a limit to what can be concluded about the $K_1$ element. If $k\in K_1(M_P)$ for some essentially normal projection $P$ on the Hilbert space $\mathcal{H}$ with a $*$-homomorphism $\sigma: C(M)\to \mathscr{L}(\mathcal{H})$, then there exists an essentially normal projection $Q\leq P$ such that $[\sigma, Q]=k\in K_1(M_P)$.

\end{document}